\documentclass{amsart}

\newtheorem{theorem}{Theorem}[section]
\newtheorem{lemma}[theorem]{Lemma}

\theoremstyle{definition}

\theoremstyle{remark}

\numberwithin{equation}{section}

\usepackage{url}
\newtheorem{proposition}{Proposition}
\usepackage{amssymb}

\begin{document}

\title{Locating a shortest vector in certain $2$-dimensional lattices}


\author{GuiXian ZOU}
\address{}
\curraddr{School of Mathematics and Statistics, Henan University, Kaifeng 475001, China}
\email{guixian\_zou@henu.edu.cn}
\thanks{}






\begin{abstract}
Let $a$, $m$ be positive integers, $1<a<m$, $\gcd(a,m)=1$. We determine the location of a shortest vector in the $2$-dimensional lattices
$$
\Lambda(a,m) = \{(x, y)\in\mathbb{Z}\times\mathbb{Z}\mid ax + y\equiv 0~(\bmod\,m)\}.
$$
This confirms a conjecture of Han Wu and Guangwu Xu.
\end{abstract}

\maketitle

\section{Introduction}

In 1247, Qin Jiushao, a great mathematician of Song Dynasty, completed  {\sl Mathematical Treatise in Nine Sections}  (see  Libbrect \cite{L1973}), in which he introduced the method of ``DaYan aggregation'' to solve a system of linear congruences. As the key step of Qin's method of ``DaYan aggregation'', the method of ``DaYan deriving one'' is created for computing a modular inverse. Specifically, for positive integers $a$ and $m$ satisfying $1<a<m$ and $\gcd(a,m)=1$, Qin described an algorithm to obtain a positive integer $u$ such that
\begin{equation}\label{eq.u}
au \equiv 1 \pmod{m}, \quad 1 < u < m.
\end{equation}

Guangwu Xu and his coauthors \cite{XL2016,Xu2018,WX2023,Xu2024} provided a faithful interpretation of Qin's algorithm in the context of modern algorithmic number theory. They established some useful properties, and revealed some unique features that are different from the extended Euclidean algorithm. Particularly, Wu and Xu \cite{WX2023} explored the application of Qin's algorithm to finding a shortest non-zero vector in the following $2$-dimensional lattices
\begin{equation}\label{eq.2dim.lattice}
\Lambda(a,m):=\{(x,y)\in\mathbb{Z}\times\mathbb{Z}\mid ax+y\equiv 0~(\bmod\,m)\},
\end{equation}
where $a,m$ are positive integers, $1<a<m$, $\gcd(a,m)=1$. Such $2$-dimensional lattices were shown to be important in the study of the GLV multiplication on elliptic curves, see for example Gallant, Lambert and Vanstone \cite{GLV2001}.

In \cite{WX2023}, the shortest vector problem of the above 2-dimensional lattices $\Lambda(a,m)$ was studied by means of the s-state  $\widehat{\mathcal{X}_k} $ of Qin's algorithm. Based on numerical experiments, it was observed that there may be a specific connection between the shortest vector of $\Lambda(a,m)$ and the row vectors of certain s-states $\{\widehat{\mathcal{X}_k}\}$.

In this paper, we reveal the further connection between Qin's algorithm and the problem of locating a shortest non-zero vector in the $2$-dimensional lattices \eqref{eq.2dim.lattice}.

To state the conclusions of the paper, we first introduce the necessary notation  and conclusions in Wu and Xu \cite{WX2023,Xu2024}.

We denote the state at the $k$-th step of Qin's algorithm as
$$
\mathcal{X}_k =
\begin{pmatrix}
x_{11}^{(k)} & x_{12}^{(k)} \\
x_{21}^{(k)} & x_{22}^{(k)}
\end{pmatrix}.
$$
The initial state is $\mathcal{X}_0 = \begin{pmatrix} 1 & a \\ 0 & m \end{pmatrix}$. Now, since $ x_{12}^{(0)} = a<x_{22}^{(0)} = m $, we perform the division algorithm to obtain the quotient $ q_1 $ and the remainder $ r_1 $:
$$
q_1 = \left\lfloor \frac{x_{22}^{(0)} - 1}{x_{12}^{(0)}} \right\rfloor, \qquad r_1 = x_{22}^{(0)} - q_1 x_{12}^{(0)},
$$
The way of updating the states is as following:
$$
x_{21}^{(1)} \leftarrow x_{21}^{(0)} + q_1 x_{11}^{(0)},\qquad
x_{22}^{(1)} \leftarrow x_{22}^{(0)} - q_1 x_{12}^{(0)}.
$$
At this point, the state  is $\mathcal{X}_1 = \begin{pmatrix} 1 & a \\ q_1 & r_1 \end{pmatrix}$. Now, since $ x_{12}^{(1)} = a > x_{22}^{(1)} = r_1 $, we perform the division algorithm to obtain the quotient $ q_2 $ and the remainder $ r_2 $:
$$
q_2 = \left\lfloor \frac{x_{12}^{(1)} - 1}{x_{22}^{(1)}} \right\rfloor, \qquad r_2 = x_{12}^{(1)} - q_2 x_{22}^{(1)},
$$
The way of updating the states is
$$
x_{11}^{(2)} \leftarrow x_{11}^{(1)} + q_2 x_{21}^{(1)},\qquad
x_{12}^{(2)} \leftarrow x_{12}^{(1)} - q_2 x_{22}^{(1)},
$$
then the state is $\mathcal{X}_2 = \begin{pmatrix} 1+ q_1q_2 & r_2 \\ q_1 & r_1 \end{pmatrix}$.

Generally, we update the second row of the states at odd steps
\begin{equation}\label{eq:odd}
x_{21}^{(k)} \leftarrow x_{21}^{(k-1)} + q_k x_{11}^{(k-1)},\qquad
x_{22}^{(k)} \leftarrow x_{22}^{(k-1)} - q_k x_{12}^{(k-1)};
\end{equation}
and update the first row of the states at even steps
\begin{equation}\label{eq:even}
x_{11}^{(k)} \leftarrow x_{11}^{(k-1)} + q_k x_{21}^{(k-1)},\qquad
x_{12}^{(k)} \leftarrow x_{12}^{(k-1)} - q_k x_{22}^{(k-1)}.
\end{equation}
The algorithm terminates when the element at the upper right corner of the state  becomes 1. Then the element at the upper left corner of the terminating state  is exactly $u$ modulo $m$ as desired in \eqref{eq.u}. We remark that, according to the updating rule of the states, except that the element at the lower left corner of the initial state is 0, all the other elements appeared in each state are positive integers.
	
For the state $\mathcal{X}_k = \begin{pmatrix} x_{11}^{(k)} & x_{12}^{(k)} \\ x_{21}^{(k)} & x_{22}^{(k)} \end{pmatrix} $, Wu and Xu \cite{WX2023} define the corresponding s-state as
\begin{equation}\label{eq.state.matrix}
\widehat{\mathcal{X}_k} = \begin{pmatrix} x_{11}^{(k)} & -x_{12}^{(k)} \\ x_{21}^{(k)} & x_{22}^{(k)} \end{pmatrix}= \begin{pmatrix} \widehat{v}_1^{(k)} \\ \widehat{v}_2^{(k)}\end{pmatrix}.
\end{equation}
	
All the following propositions were established in \cite[p.\,7, pp.\,13--17]{WX2023}.

\begin{proposition}\label{prop1}
Given the initial state $\mathcal{X}_0 = \begin{pmatrix} 1 & a \\ 0 & m \end{pmatrix}$, we have
$$
\widehat{\mathcal{X}_k} =
\begin{cases}
\begin{pmatrix} 1 & 0 \\q_k & 1 \end{pmatrix} \widehat{\mathcal{X}_{k-1}},
& \text{if } k \text{ is odd}, \\[1em]
\begin{pmatrix} 1 & 0 \\q_k & 1 \end{pmatrix}^\top \widehat{\mathcal{X}_{k-1}},
& \text{if } k \text{ is even}.
\end{cases}
$$
where $q_k$ is the quotient computed in the $k$-th step of Qin's algorithm.
\end{proposition}

Proposition \ref{prop1} can be restated as follows:

When $k$ is odd, $\widehat{v}_2^{(k)} = q_k\widehat{v}_1^{(k-1)} + \widehat{v}_2^{(k-1)}$, $\widehat{v}_1^{(k)} = \widehat{v}_1^{(k-1)}$;

When $k$ is even, $\widehat{v}_1^{(k)} = q_k\widehat{v}_2^{(k-1)} + \widehat{v}_1^{(k-1)}$, $\widehat{v}_2^{(k)} = \widehat{v}_2^{(k-1)}$.

For the convenience of subsequent proofs, we define $\delta_k\in\{0,1\}$ by
$$
\delta_k\equiv k~(\bmod\,2).
$$
Then we have
\begin{equation}\label{eq.vector update}
\widehat{v}_{1+\delta_k}^{(k)} = \widehat{v}_{1+\delta_k}^{(k-1)} + q_k \widehat{v}_{2-\delta_k}^{(k-1)}, \qquad \widehat{v}_{2-\delta_k}^{(k)} = \widehat{v}_{2-\delta_k}^{(k-1)}.
\end{equation}

For any vector $(x,y)\in\mathbb{Z}\times\mathbb{Z}$ in the 2-dimensional lattices $\Lambda(a,m)$, we define $\|(x,y)\|^2=x^2+y^2$, i.e., the square of the distance from the lattice point $(x,y)$ to the origin, which will be directly referred to as the squared length of the vector $(x,y)$ in what follows. Denote the inner product of the row vectors of the s-state  $\widehat{\mathcal{X}_k}$ by
\begin{equation}\label{eq.inner.product}
\mathcal{I}_k = \langle \widehat{v}_1^{(k)}, \widehat{v}_2^{(k)} \rangle = x_{11}^{(k)} x_{21}^{(k)} - x_{12}^{(k)} x_{22}^{(k)}.
\end{equation}

\begin{proposition}\label{prop2}
Each s-state  $ \widehat{\mathcal{X}_k} $ of Qin's algorithm is a basis of $ \Lambda(a,m) $. In particular, the volume of the lattice $ \Lambda(a,m) $ is exactly $ m $.
\end{proposition}
	
\begin{proposition}\label{prop3}
There exists an s-state  $\widehat{\mathcal{X}_k} = \begin{pmatrix} \widehat{v}_1^{(k)} \\ \widehat{v}_2^{(k)} \end{pmatrix}$ such that the following set
$$
\left\{ \widehat{v}_1^{(k)}, \widehat{v}_2^{(k)}, \widehat{v}_1^{(k)} + \widehat{v}_2^{(k)}, \widehat{v}_1^{(k)} - \widehat{v}_2^{(k)} \right\}
$$
contains the shortest vector of $\Lambda(a,m)$.
\end{proposition}

According to Proposition \ref{prop3}, Qin's algorithm guarantees that the shortest non-zero vector could be found from the states. Clearly, brute-force search is inefficient. Therefore, it is naturally to consider a suitable indicator, then by which we can locate effectively a shortest non-zero vector in the states. We shall prove that the inner product is exactly such an effective indicator.
	
\begin{proposition}\label{prop4}
For any $k>0$, we have
\begin{equation}\label{eq.recursive formula}
\mathcal{I}_k = \mathcal{I}_{k-1} + q_k \| \widehat{v}_{2-\delta_k}^{(k)} \|^2 = \mathcal{I}_{k-1} + q_k \| \widehat{v}_{2-\delta_k}^{(k-1)} \|^2.
\end{equation}
\end{proposition}
	
Proposition \ref{prop4} gives a recursive formula for $\mathcal{I}_k$ and tells us that ${\mathcal{I}_k} $ is monotonically increasing. By definition, $\mathcal{I} _0=-am<0$. If the inner product of the terminating s-state  is greater than 0, then there exist two s-states $\widehat{\mathcal{X}_{k_0}}$ and $\widehat{\mathcal{X}_{k_0+1}}$ satisfying $\mathcal{I}_{k_0}<0$ and $\mathcal{I}_{k_0+1}\geqslant0$; of course, it is also possible that the inner product does not change sign throughout the process. In this case, we can quickly determine the shortest vector.

\begin{proposition}\label{prop5}
If the inner product does not change sign, then $(a^{-1}, -1)$ is a shortest non-zero vector in $\Lambda(a,m)$.
\end{proposition}

This proposition follows from \cite[pp.\,16--17]{WX2023}.
	
The purpose of this paper is to prove a conjecture proposed by Wu and Xu (see \cite[p.\,17]{WX2023}) and to give a necessary and sufficient condition for determining the exact position of the shortest non-zero vector.
	
\begin{theorem}\label{th:main}
When the inner product changes sign, the shortest non-zero vector of the $2$-dimensional lattices $\Lambda(a,m)$ must lie in one of the rows of the two s-states  $\widehat{\mathcal{X}_{k_0}}$ and $\widehat{\mathcal{X}_{k_0+1}}$ satisfying $\mathcal{I}_{k_0}<0$ and $\mathcal{I}_{k_0+1}\geqslant0$. When the inner product does not change sign, the shortest non-zero vector of the $2$-dimensional lattice $\Lambda(a,m)$ must lie in the first row of the terminating s-state  $\widehat{\mathcal{X}_N}$.
\end{theorem}

It must be noted that a shortest non-zero vector must appear in the case of Theorem \ref{th:main}, and there may also other  shortest non-zero vectors appear in other s-states. The conclusion of Theorem \ref{th:main} when the inner product does not change sign is exactly Proposition \ref{prop5}, so only the case where the inner product changes sign needs to be considered.

\section{Lemmata}

\begin{lemma}\label{lem1}
For any $k>0$, we have
\begin{equation}\label{eq.1}
\| \widehat{v}_{1+\delta_k}^{(k)} \|^2 = \| \widehat{v}_{1+\delta_k}^{(k-1)} \|^2 + q_k ( \mathcal{I}_{k-1} + \mathcal{I}_k ).
\end{equation}
\end{lemma}

\begin{proof}
By applying  \eqref{eq.vector update} and \eqref{eq.recursive formula}, we obtain
\begin{align*}
\|\widehat{v}_{1+\delta_{k}}^{(k)}\|^2
& = \|q_k\widehat{v}_{2-\delta_{k}}^{(k-1)} + \widehat{v}_{1+\delta_{k}}^{(k-1)}\|^2
= q_k^2\|\widehat{v}_{2-\delta_{k}}^{(k-1)}\|^2 + \|\widehat{v}_{1+\delta_{k}}^{(k-1)}\|^2 + 2q_k \mathcal{I}_{k-1} . \\
&=\|\widehat{v}_{1+\delta_{k}}^{(k-1)}\|^2+q_k(q_k\|\widehat{v}_{2-\delta_{k}}^{(k-1)}\|^2+\mathcal{I}_{k-1})+q_k\mathcal{I}_{k-1}\\
&=\| \widehat{v}_{1+\delta_k}^{(k-1)} \|^2 + q_k ( \mathcal{I}_{k-1} + \mathcal{I}_k ).
\qedhere
\end{align*}
\end{proof}

\begin{lemma}\label{lem3}
For any $0<k\leqslant N$, if $q_k = 1$, then we have
\begin{equation}\label{eq.main1}
\| \widehat{v}_1^{(k-1)} + \widehat{v}_2^{(k-1)} \| = \| \widehat{v}_{1+\delta_k}^{(k)} \|, \quad \| \widehat{v}_1^{(k)} - \widehat{v}_2^{(k)} \| = \| \widehat{v}_{1+\delta_k}^{(k-1)} \|.
\end{equation}
If $q_k \geqslant 2$, then we have
\begin{equation}\label{eq.main2}
\| \widehat{v}_1^{(k-1)} + \widehat{v}_2^{(k-1)} \| > \| \widehat{v}_{2-\delta_k}^{(k-1)} \|, \quad \| \widehat{v}_1^{(k)} - \widehat{v}_2^{(k)} \| > \| \widehat{v}_{2-\delta_k}^{(k)} \|.
\end{equation}
Meanwhile, $\| \widehat{v}_1^{(N)} + \widehat{v}_2^{(N)} \| > \| \widehat{v}_{2-\delta_k}^{(N)} \|$, $\| \widehat{v}_1^{(0)} - \widehat{v}_2^{(0)} \| > \| \widehat{v}_{2-\delta_k}^{(0)} \|$ is also true.
\end{lemma}

\begin{proof}
From formula \eqref{eq.vector update}, it follows that when $q_k=1$,
$$
\widehat{v}_{1+\delta_k}^{(k-1)} + \widehat{v}_{2-\delta_k}^{(k-1)}=\widehat{v}_{1+\delta_k}^{(k)},\qquad \widehat{v}_{2-\delta_k}^{(k)} = \widehat{v}_{2-\delta_k}^{(k-1)}.
$$
It is easy to deduce the above two norm relations.

When $q_k \geqslant 2$,

1. $\| \widehat{v}_1^{(k-1)} + \widehat{v}_2^{(k-1)} \|^2-\| \widehat{v}_{2-\delta_k}^{(k-1)} \|^2=\|\widehat{v}_{1+\delta_k}^{(k-1)}\|^2+2\mathcal{I}_{k-1}$.

From the definitions of the length and the inner product, we obtain
\begin{align*}
\|\widehat{v}_{1+\delta_k}^{(k-1)}\|^2+2\mathcal{I}_{k}
&=x_{{1+\delta_k},1}^{(k-1)^{2}}+x_{{1+\delta_k},2}^{(k-1)^{2}}+2(x_{11}^{(k-1)}x_{21}^{(k-1)} -x_{12}^{(k-1)} x_{22}^{(k-1)})\\
&=x_{{1+\delta_k},1}^{(k-1)}(x_{{1+\delta_k}1}^{(k-1)}+2x_{{2-\delta_k},1}^{(k-1)})+x_{{1+\delta_k},2}^{(k-1)}(x_{{1+\delta_k},2}^{(k-1)}-2x_{{2-\delta_k},2}^{(k-1)}),
\end{align*}
Since $q_k\geqslant2$,
$$
x_{{1+\delta_k},2}^{(k-1)}-2x_{{2-\delta_k},2}^{(k-1)}\geqslant x_{{1+\delta_k},2}^{(k-1)}-q_{k} x_{{2-\delta_k},2}^{(k-1)}=x_{{1+\delta_k},2}^{(k)}>0,
$$
hence $\| \widehat{v}_1^{(k-1)} + \widehat{v}_2^{(k-1)} \| > \| \widehat{v}_{2-\delta_k}^{(k-1)} \|$.

2. $\| \widehat{v}_1^{(k)} - \widehat{v}_2^{(k)} \|^2-\| \widehat{v}_{2-\delta_k}^{(k)} \|^2=\|\widehat{v}_{1+\delta_k}^{(k)}\|^2-2\mathcal{I}_{k}$.

From the definitions of the length and the inner product, we obtain
\begin{align*}
\|\widehat{v}_{1+\delta_k}^{(k)}\|^2-2\mathcal{I}_{k}
&=x_{{1+\delta_k},1}^{(k)^{2}}+x_{{1+\delta_k},2}^{(k)^{2}}-2(x_{11}^{(k)}x_{21}^{(k)} -x_{12}^{(k)} x_{22}^{(k)})\\
&=x_{{1+\delta_k},1}^{(k)}(x_{{1+\delta_k}1}^{(k)}-2x_{{2-\delta_k},1}^{(k)})+x_{{1+\delta_k},2}^{(k)}(x_{{1+\delta_k},2}^{(k)}+2x_{{2-\delta_k},2}^{(k)}),
\end{align*}
Since $q_k\geqslant2$,
$$
x_{{1+\delta_k},1}^{(k)}-2x_{{2-\delta_k},1}^{(k)}\geqslant x_{{1+\delta_k},1}^{(k)}-q_{k} x_{{2-\delta_k},1}^{(k)}=x_{{1+\delta_k},1}^{(k-1)}>0,
$$
hence $\| \widehat{v}_1^{(k)} - \widehat{v}_2^{(k)} \| > \| \widehat{v}_{2-\delta_k}^{(k)} \|$.

Meanwhile
$$
\| \widehat{v}_1^{(N)} + \widehat{v}_2^{(N)} \|^2 = \| \widehat{v}_1^{(N)} \|^2+\| \widehat{v}_2^{(N)}\|^2+2\mathcal{I}_N> \| \widehat{v}_{2-\delta_k}^{(N)} \|^2,
$$
$$
\| \widehat{v}_1^{(0)} - \widehat{v}_2^{(0)} \|^2 = \| \widehat{v}_1^{(0)} \|^2+\| \widehat{v}_2^{(0)}\|^2-2\mathcal{I}_0> \| \widehat{v}_{2-\delta_k}^{(0)} \|^2,
$$
This implies that $\| \widehat{v}_1^{(N)} + \widehat{v}_2^{(N)} \| > \| \widehat{v}_{2-\delta_k}^{(N)} \|$, \quad $\| \widehat{v}_1^{(0)} - \widehat{v}_2^{(0)} \| > \| \widehat{v}_{2-\delta_k}^{(0)} \|.$
\end{proof}

\section{Proof of Theorem \ref{th:main}}

In this section, we use the above conclusions to give the proof of Theorem \ref{th:main}.

\begin{proof}[Proof of Theorem  \ref{th:main}]
First, by Lemma \ref{lem3}, we have
$$
\|\widehat{v}_1^{(k)} + \widehat{v}_2^{(k)}\|, \|\widehat{v}_1^{(k)} - \widehat{v}_2^{(k)}\|\geqslant \min_{0\leqslant k \leqslant N}\{ \| \widehat{v}_1^{(k)} \|, \| \widehat{v}_2^{(k)} \| \},
\qquad 0\leqslant k \leqslant N.
$$	
Then, by Proposition \ref{prop3}, it follows that the shortest non-zero vector must be a row vector of some states. We make the following definitions.
For $k\leqslant k_0$, let $ L_k = \min\{ \| \widehat{v}_1^{(k)} \|, \| \widehat{v}_2^{(k)} \| \} $. For $k>k_0$, let $R_k=\min\{ \| \widehat{v}_1^{(k)} \|, \| \widehat{v}_2^{(k)} \| \}.$

We consider the following two cases, combined with Lemma \ref{lem1}.

(i) Suppose $ L_{k-1} = \| \widehat{v}_{2-\delta_k}^{(k-1)} \| $. Then we have
$$
L_{k-1} = \| \widehat{v}_{2-\delta_k}^{(k-1)} \| = \| \widehat{v}_{2-\delta_k}^{(k)} \| \geqslant L_k,
$$

(ii) Suppose $ L_{k-1} = \| \widehat{v}_{1+\delta_k}^{(k-1)} \| $. Then we have
$$
L_{k-1}^2 = \| \widehat{v}_{1+\delta_k}^{(k-1)} \|^2 = \| \widehat{v}_{1+\delta_k}^{(k)} \|^2 - q_k(\mathcal{I}_{k-1} + \mathcal{I}_k) > \| \widehat{v}_{1+\delta_k}^{(k)} \|^2 \geqslant L_k^2,
$$
which implies $L_{k-1} \geqslant L_{k}.$

Similarly, we have $R_k \geqslant R_{k-1}.$

(i) Suppose $ R_k = \| \widehat{v}_{1+\delta_k}^{(k)} \| $. Then we have
$$
R_k^2 = \| \widehat{v}_{1+\delta_k}^{(k)} \|^2 = \| \widehat{v}_{1+\delta_k}^{(k-1)} \|^2 + q_k(\mathcal{I}_{k-1} + \mathcal{I}_k) > \| \widehat{v}_{1+\delta_k}^{(k-1)} \|^2 \geqslant R_{k-1}^2,
$$

(ii) Suppose $ R_k = \| \widehat{v}_{2-\delta_k}^{(k)} \| $. Then we have
$$
R_k = \| \widehat{v}_{2-\delta_k}^{(k)} \| = \| \widehat{v}_{2-\delta_k}^{(k-1)} \| \geqslant R_{k-1}.
$$
From these two monotonic relations, we complete the proof of Theorem 1.
\end{proof}

\section{Determining the Exact Position of the Shortest Non-zero Vector}

First, we calculate $\| \widehat{v}_{1+\delta_{K+1}}^{(K+1)} \|^2-\| \widehat{v}_{2-\delta_{K+1}}^{(K+1)} \|^2$.
By Lemma \ref{lem1}, we obtain
\begin{align*}
\| \widehat{v}_{1+\delta_{K+1}}^{(K+1)} \|^2-\| \widehat{v}_{2-\delta_{K+1}}^{(K+1)} \|^2 &= \| \widehat{v}_{1+\delta_{K+1}}^{(K)} \|^2 + q_{K+1} ( \mathcal{I}_{K} + \mathcal{I}_{K+1} )-\| \widehat{v}_{2-\delta_{K+1}}^{(K)} \|^2\\
&=-(\| \widehat{v}_{1+\delta_{K}}^{(K)} \|^2-\| \widehat{v}_{2-\delta_{K}}^{(K)} \|^2)+q_{K+1} ( \mathcal{I}_{K} + \mathcal{I}_{K+1} ),
\end{align*}
which is a recursive formula. Applying this recursion repeatedly, we can obtain
\[
\begin{split}
\| \widehat{v}_{1+\delta_{K+1}}^{(K+1)} \|^2-\| \widehat{v}_{2-\delta_{K+1}}^{(K+1)} \|^2
&= q_{K+1}\mathcal{I}_{K+1}+\sum_{k=0}^{K}(-1)^{K-k}(q_{k+1}-q_{k})\mathcal{I}_{k} \\
&\quad + (-1)^{K}(a^2+1-m^2).
\end{split}
\]
where $q_{0}=0$.

We define
$$
D_{K+1}=q_{K+1}\mathcal{I}_{K+1}+\sum_{k=0}^{K}(-1)^{K-k}(q_{k+1}-q_{k})\mathcal{I}_{k}+(-1)^{K}(a^2+1-m^2).
$$
Based on this, we can give a method for determining the exact position of the shortest non-zero vector.

\begin{theorem}\label{th.main2}
If $|\mathcal{I}_{k_0}|\geqslant|\mathcal{I}_{k_0+1}|$,

(a) when $D_{k_0+1}\geqslant0$, the shortest non-zero vector is $\widehat{v}_{2-\delta_{k_0+1}}^{(k_0+1)}$;

(b) when $D_{k_0+1}\leqslant0$, the shortest non-zero vector is $\widehat{v}_{1+\delta_{k_0+1}}^{(k_0+1)}$.

If $|\mathcal{I}_{k_0}|\leqslant|\mathcal{I}_{k_0+1}|$,

(a) when $D_{k_0}\geqslant0$, the shortest non-zero vector is $\widehat{v}_{2-\delta_{k_0}}^{(k_0)}$;

(b) when $D_{k_0}\leqslant0$, the shortest non-zero vector is $\widehat{v}_{1+\delta_{k_0}}^{(k_0)}$.
\end{theorem}

\begin{proof}
The proof of the theorem mainly uses Lemma \ref{lem1}.

When $|\mathcal{I}_{k_0}|\geqslant|\mathcal{I}_{k_0+1}|$, this implies that $\mathcal{I}_{k_0}+\mathcal{I}_{k_0+1}\leqslant 0.$

From Lemma \ref{lem1} and  \eqref{eq.vector update}, we have
$$
\| \widehat{v}_{1+\delta_{k_0+1}}^{(k_0+1)} \|^2 = \| \widehat{v}_{1+\delta_{k_0+1}}^{(k_0)} \|^2 + q_{k_0+1} ( \mathcal{I}_{k_0} + \mathcal{I}_{k_0+1} )\leqslant \| \widehat{v}_{1+\delta_{k_0+1}}^{(k_0)} \|^2,
$$
$$
\|\widehat{v}_{2-\delta_{k_0+1}}^{(k_0+1)}\| = \|\widehat{v}_{2-\delta_{k_0+1}}^{(k_0)}\|,
$$
hence $R_{k_0+1}\leqslant L_{k_0}$.

Furthermore, we have
$$
\| \widehat{v}_{1+\delta_{k_0+1}}^{(k_0+1)} \|^2-\| \widehat{v}_{2-\delta_{k_0+1}}^{(k_0+1)} \|^2 =D_{k_0+1}.
$$
Then we prove the first half of the theorem.

When $|\mathcal{I}_{k_0}|\leqslant|\mathcal{I}_{k_0+1}|$, this implies that $\mathcal{I} _{k_0}+\mathcal{I}_{k_0+1}\geqslant 0.$

From Lemma \ref{lem1} and  \eqref{eq.vector update}, we obtain
$$
\| \widehat{v}_{1+\delta_{k_0+1}}^{(k_0+1)} \|^2 = \| \widehat{v}_{1+\delta_{k_0+1}}^{(k_0)} \|^2 + q_{k_0+1} ( \mathcal{I}_{k_0} + \mathcal{I}_{k_0+1} )\geqslant \| \widehat{v}_{1+\delta_{k_0+1}}^{(k_0)} \|^2,
$$
$$
\|\widehat{v}_{2-\delta_{k_0+1}}^{(k_0+1)}\| = \|\widehat{v}_{2-\delta_{k_0+1}}^{(k_0)}\|,
$$
hence $R_{k_0+1}\geqslant L_{k_0}$.

Similarly, we have
$$
\| \widehat{v}_{1+\delta_{k_0}}^{(k_0)} \|^2-\| \widehat{v}_{2-\delta_{k_0}}^{(k_0)} \|^2 =D_{k_0}.
$$
Then the second half of the theorem is proved.
\end{proof}

\section*{Acknowledgement}

I am grateful to Professor Guangwu Xu for his valuable suggestions which leads to a simplification on the proof of a preliminary manuscript. I thank Dr. Ke Gong for his guidance and many helps during the preparation of the present work. 


\end{document}